\documentclass[draft,11pt]{amsart}
\usepackage{amsmath}
\usepackage{amssymb,latexsym}
\usepackage{amscd}
\usepackage{xspace}
\usepackage{verbatim}

\newtheorem{theorem}{Theorem}[section]
\newtheorem{lemma}{Lemma}[section]
\newtheorem{proposition}{Proposition}[section]
\newtheorem{corollary}{Corollary}[section]

\theoremstyle{definition}

\begin{document}
\numberwithin{equation}{section}
\title[Fading absorption]{Fading absorption in non-linear elliptic equations}
\author{Moshe Marcus }
\thanks{This research was supported by The Israel Science Foundation grant No. 91/10}
\address{Department of Mathematics, Technion
Haifa, ISRAEL}
\email{marcusm@math.technion.ac.il}

\author{ Andrey Shishkov}
\address{Institute of Appl. Math. and  Mech., NAS of Ukraine, Donetsk, Ukraine}
\email{shishkov@iamm.ac.donetsk.ua}
\date{\today}

\newcommand{\txt}[1]{\;\text{ #1 }\;}
\newcommand{\tbf}{\textbf}
\newcommand{\tit}{\textit}
\newcommand{\tsc}{\textsc}
\newcommand{\trm}{\textrm}
\newcommand{\mbf}{\mathbf}
\newcommand{\mrm}{\mathrm}
\newcommand{\bsym}{\boldsymbol}
\newcommand{\scs}{\scriptstyle}
\newcommand{\sss}{\scriptscriptstyle}
\newcommand{\txts}{\textstyle}
\newcommand{\dsps}{\displaystyle}
\newcommand{\fnz}{\footnotesize}
\newcommand{\scz}{\scriptsize}
\newcommand{\be}{\begin{equation}}
\newcommand{\bel}[1]{\begin{equation}\label{#1}}
\newcommand{\ee}{\end{equation}}
\newtheorem{subn}{\name}
\renewcommand{\thesubn}{}
\newcommand{\bsn}[1]{\def\name{#1$\!\!$}\begin{subn}}
\newcommand{\esn}{\end{subn}}
\newtheorem{sub}{\name}[section]
\newcommand{\dn}[1]{\def\name{#1}}   
\newcommand{\bs}{\begin{sub}}
\newcommand{\es}{\end{sub}}
\newcommand{\bsl}[1]{\begin{sub}\label{#1}}
\newcommand{\bth}[1]{\def\name{Theorem}\begin{sub}\label{t:#1}}
\newcommand{\blemma}[1]{\def\name{Lemma}\begin{sub}\label{l:#1}}
\newcommand{\bcor}[1]{\def\name{Corollary}\begin{sub}\label{c:#1}}
\newcommand{\bdef}[1]{\def\name{Definition}\begin{sub}\label{d:#1}}
\newcommand{\bprop}[1]{\def\name{Proposition}\begin{sub}\label{p:#1}}
\newcommand{\bnote}[1]{\def\name{\mdseries\scshape Notation}\begin{sub}\label{n:#1}}
\newcommand{\bproof}{\begin{proof}}
\newcommand{\eproof}{\end{proof}}
\newcommand{\bcom}{}
\newcommand{\req}{\eqref}
\newcommand{\rth}[1]{Theorem~\ref{t:#1}}
\newcommand{\rlemma}[1]{Lemma~\ref{l:#1}}
\newcommand{\rcor}[1]{Corollary~\ref{c:#1}}
\newcommand{\rdef}[1]{Definition~\ref{d:#1}}
\newcommand{\rprop}[1]{Proposition~\ref{p:#1}}
\newcommand{\rnote}[1]{Notation~\ref{n:#1}}
\newcommand{\BA}{\begin{array}}
\newcommand{\EA}{\end{array}}
\newcommand{\BAN}{\renewcommand{\arraystretch}{1.2}
\setlength{\arraycolsep}{2pt}\begin{array}}
\newcommand{\BAV}[2]{\renewcommand{\arraystretch}{#1}
\setlength{\arraycolsep}{#2}\begin{array}}
\newcommand{\BSA}{\begin{subarray}}
\newcommand{\ESA}{\end{subarray}}
\newcommand{\BAL}{\begin{aligned}}
\newcommand{\EAL}{\end{aligned}}
\newcommand{\BALG}{\begin{alignat}}
\newcommand{\EALG}{\end{alignat}}
\newcommand{\BALGN}{\begin{alignat*}}
\newcommand{\EALGN}{\end{alignat*}}
\newcommand{\note}[1]{\noindent\textit{#1.}\hspace{2mm}}
\newcommand{\Remark}{\note{Remark}}
\newcommand{\forevery}{\quad \forall}
\newcommand{\1}{\\[1mm]}
\newcommand{\2}{\\[2mm]}
\newcommand{\3}{\\[3mm]}
\newcommand{\set}[1]{\{#1\}}
\def\({{\rm (}}
\def\){{\rm )}}
\newcommand{\st}[1]{{\rm (#1)}}
\newcommand{\lra}{\longrightarrow}
\newcommand{\lla}{\longleftarrow}
\newcommand{\llra}{\longleftrightarrow}
\newcommand{\Lra}{\Longrightarrow}
\newcommand{\Lla}{\Longleftarrow}
\newcommand{\Llra}{\Longleftrightarrow}
\newcommand{\warrow}{\rightharpoonup}
\def\dar{\downarrow}
\def\uar{\uparrow}
\newcommand{\paran}[1]{\left (#1 \right )}
\newcommand{\sqrbr}[1]{\left [#1 \right ]}
\newcommand{\curlybr}[1]{\left \{#1 \right \}}
\newcommand{\absol}[1]{\left |#1\right |}
\newcommand{\norm}[1]{\left \|#1\right \|}
\newcommand{\angbr}[1]{\left< #1\right>}
\newcommand{\paranb}[1]{\big (#1 \big )}
\newcommand{\sqrbrb}[1]{\big [#1 \big ]}
\newcommand{\curlybrb}[1]{\big \{#1 \big \}}
\newcommand{\absolb}[1]{\big |#1\big |}
\newcommand{\normb}[1]{\big \|#1\big \|}
\newcommand{\angbrb}[1]{\big\langle #1 \big \rangle}
\newcommand{\thkl}{\rule[-.5mm]{.3mm}{3mm}}
\newcommand{\thknorm}[1]{\thkl #1 \thkl\,}
\newcommand{\trinorm}[1]{|\!|\!| #1 |\!|\!|\,}
\newcommand{\vstrut}[1]{\rule{0mm}{#1}}
\newcommand{\rec}[1]{\frac{1}{#1}}
\newcommand{\opname}[1]{\mathrm{#1}\,}
\newcommand{\supp}{\opname{supp}}
\newcommand{\dist}{\opname{dist}}
\newcommand{\sign}{\opname{sign}}
\newcommand{\diam}{\opname{diam}}
\newcommand{\q}{\quad}
\newcommand{\qq}{\qquad}
\newcommand{\hsp}[1]{\hspace{#1mm}}
\newcommand{\vsp}[1]{\vspace{#1mm}}
\newcommand{\prt}{\partial}
\newcommand{\sms}{\setminus}
\newcommand{\ems}{\emptyset}
\newcommand{\ti}{\times}
\newcommand{\pr}{^\prime}
\newcommand{\ppr}{^{\prime\prime}}
\newcommand{\tl}{\tilde}
\newcommand{\wtl}{\widetilde}
\newcommand{\sbs}{\subset}
\newcommand{\Sbs}{\Subset}
\newcommand{\sbeq}{\subseteq}
\newcommand{\ovl}{\overline}
\newcommand{\unl}{\underline}
\newcommand{\nin}{\not\in}
\newcommand{\pfrac}[2]{\genfrac{(}{)}{}{}{#1}{#2}}
\newcommand{\tin}{\to\infty}
\newcommand{\ind}[1]{_{\scriptscriptstyle #1}}
\newcommand{\chr}[1]{\chi\ind{#1}}
\newcommand{\rest}[1]{\big |\ind{#1}}
\newcommand{\Sol}[2]{\mathrm{Sol}\ind{#2}{#1}}
\newcommand{\wkc}{weak convergence\xspace}
\newcommand{\wrto}{with respect to\xspace}
\newcommand{\cons}{consequence\xspace}
\newcommand{\consy}{consequently\xspace}
\newcommand{\Consy}{Consequently\xspace}
\newcommand{\Essy}{Essentially\xspace}
\newcommand{\essy}{essentially\xspace}
\newcommand{\mnz}{minimizer\xspace}
\newcommand{\sth}{such that\xspace}
\newcommand{\ngh}{neighborhood\xspace}
\newcommand{\nghs}{neighborhoods\xspace}
\newcommand{\seq}{sequence\xspace}
\newcommand{\seqs}{sequences\xspace}
\newcommand{\sseq}{subsequence\xspace}
\newcommand{\ifif}{if and only if\xspace}
\newcommand{\suff}{sufficiently\xspace}
\newcommand{\abc}{absolutely continuous\xspace}
\newcommand{\sol}{solution\xspace}
\newcommand{\subss}{sub-solutions\xspace}
\newcommand{\subs}{sub-solution\xspace}
\newcommand{\supers}{super-solution\xspace}
\newcommand{\superss}{super-solutions  \xspace}
\newcommand{\Wlg}{Without loss of generality\xspace}
\newcommand{\wlg}{without loss of generality\xspace}
\newcommand{\locun}{locally uniformly\xspace}
\newcommand{\bvp}{boundary value problem\xspace}
\newcommand{\bvps}{boundary value problems\xspace}
\newcommand{\bdw}{\partial\Gw}
\newcommand{\Capq}{C_{2/q,q'}}
\newcommand{\Wq}{W^{2/q,q'}}
\newcommand{\Wqdual}{W^{-2/q,q}}
\newcommand{\Wqdb}{W^{-2/q,q}_+(\bdw)}
\newcommand{\sbsq}{\overset{q}{\sbs}}
\newcommand{\app}[1]{\underset{#1}{\approx}}
\newcommand{\suppq}{\mathrm{supp}^q_{\bdw}\,}
\newcommand{\convq}{\overset{q}{\to}}
\newcommand{\barq}[1]{\bar{#1}^{^q}}
\newcommand{\prtq}{\partial_q}
\newcommand{\tr}{\mathrm{tr}\,}
\newcommand{\Tr}{\mathrm{Tr}\,}
\newcommand{\trin}[1]{\mathrm{tr}\ind{#1}}
\newcommand{\Ltrin}[1]{\text{\rm L-tr}\ind{#1}}
\newcommand{\Mtrin}[1]{\text{\rm M-tr}\ind{#1}}
\newcommand{\qcl}{$q$-closed\xspace}
\newcommand{\qop}{$q$-open\xspace}
\newcommand{\gsmod}{$\gs$-moderate\xspace}
\newcommand{\gsreg}{$\gs$-regular\xspace}
\newcommand{\qreg}{quasi regular\xspace}
\newcommand{\qeq}{$q$-equivalent\xspace}
\newcommand{\ppf}{\underset{f}{\prec\prec}}
\newcommand{\ofrown}{\overset{\frown}}
\newcommand{\modcon}{\underset{mod}{\lra}}
\newcommand{\ugb}[1]{u\chr{\Gs_\gb(#1)}}
\def\qsupp{q\text{-supp}\,}
\def\RN{\BBR^N}
\def\loc{\ind{\rm loc}}
\def\qloc{\ind{\ell(2,q')}}
\def\bcom{}
\def\ga{\alpha}     \def\gb{\beta}       \def\gg{\gamma}
\def\gc{\chi}       \def\gd{\delta}      \def\ge{\epsilon}
\def\gth{\theta}                         \def\vge{\varepsilon}
\def\gf{\phi}       \def\vgf{\varphi}    \def\gh{\eta}
\def\gi{\iota}      \def\gk{\kappa}      \def\gl{\lambda}
\def\gm{\mu}        \def\gn{\nu}         \def\gp{\pi}
\def\vgp{\varpi}    \def\gr{\rho}        \def\vgr{\varrho}
\def\gs{\sigma}     \def\vgs{\varsigma}  \def\gt{\tau}
\def\gu{\upsilon}   \def\gv{\vartheta}   \def\gw{\omega}
\def\gx{\xi}        \def\gy{\psi}        \def\gz{\zeta}
\def\Gg{\Gamma}     \def\Gd{\Delta}      \def\Gf{\Phi}
\def\Gth{\Theta}
\def\Gl{\Lambda}    \def\Gs{\Sigma}      \def\Gp{\Pi}
\def\Gw{\Omega}     \def\Gx{\Xi}         \def\Gy{\Psi}

\def\CS{{\mathcal S}}   \def\CM{{\mathcal M}}   \def\CN{{\mathcal N}}
\def\CR{{\mathcal R}}   \def\CO{{\mathcal O}}   \def\CP{{\mathcal P}}
\def\CA{{\mathcal A}}   \def\CB{{\mathcal B}}   \def\CC{{\mathcal C}}
\def\CD{{\mathcal D}}   \def\CE{{\mathcal E}}   \def\CF{{\mathcal F}}
\def\CG{{\mathcal G}}   \def\CH{{\mathcal H}}   \def\CI{{\mathcal I}}
\def\CJ{{\mathcal J}}   \def\CK{{\mathcal K}}   \def\CL{{\mathcal L}}
\def\CT{{\mathcal T}}   \def\CU{{\mathcal U}}   \def\CV{{\mathcal V}}
\def\CZ{{\mathcal Z}}   \def\CX{{\mathcal X}}   \def\CY{{\mathcal Y}}
\def\CW{{\mathcal W}}
\def\BBA {\mathbb A}   \def\BBb {\mathbb B}    \def\BBC {\mathbb C}
\def\BBD {\mathbb D}   \def\BBE {\mathbb E}    \def\BBF {\mathbb F}
\def\BBG {\mathbb G}   \def\BBH {\mathbb H}    \def\BBI {\mathbb I}
\def\BBJ {\mathbb J}   \def\BBK {\mathbb K}    \def\BBL {\mathbb L}
\def\BBM {\mathbb M}   \def\BBN {\mathbb N}    \def\BBO {\mathbb O}
\def\BBP {\mathbb P}   \def\BBR {\mathbb R}    \def\BBS {\mathbb S}
\def\BBT {\mathbb T}   \def\BBU {\mathbb U}    \def\BBV {\mathbb V}
\def\BBW {\mathbb W}   \def\BBX {\mathbb X}    \def\BBY {\mathbb Y}
\def\BBZ {\mathbb Z}

\def\GTA {\mathfrak A}   \def\GTB {\mathfrak B}    \def\GTC {\mathfrak C}
\def\GTD {\mathfrak D}   \def\GTE {\mathfrak E}    \def\GTF {\mathfrak F}
\def\GTG {\mathfrak G}   \def\GTH {\mathfrak H}    \def\GTI {\mathfrak I}
\def\GTJ {\mathfrak J}   \def\GTK {\mathfrak K}    \def\GTL {\mathfrak L}
\def\GTM {\mathfrak M}   \def\GTN {\mathfrak N}    \def\GTO {\mathfrak O}
\def\GTP {\mathfrak P}   \def\GTR {\mathfrak R}    \def\GTS {\mathfrak S}
\def\GTT {\mathfrak T}   \def\GTU {\mathfrak U}    \def\GTV {\mathfrak V}
\def\GTW {\mathfrak W}   \def\GTX {\mathfrak X}    \def\GTY {\mathfrak Y}
\def\GTZ {\mathfrak Z}   \def\GTQ {\mathfrak Q}
\font\Sym= msam10
\def\SYM#1{\hbox{\Sym #1}}
\def\bmn{\mathbf{n}}
\def\bma{\mathbf{a}}
\def\prtn{\prt_{\bmn}}

\def\txin{\txt{in}}
\def\txon{\txt{on}}
\def\C2q{C_{2,q'}}
\def\W2q{W^{2,q'}}
\def\w02q{W_{0,\infty}^{2,q'}(D)}
\def\qfine{\text{$\C2q$-finely}\xspace}
\def\qsupp{\mathrm{supp}\ind{(2,q')}}
\def\qcomp{\text{$\C2q$-pseudo compact}\xspace}
\def\qae{\text{$\C2q$ a.e.}\xspace}
\def\qinterior{\mathrm{int}_q}
\def\limsim{\overset{\scriptscriptstyle lim}{\sim}}
\def\limsbs{\overset{\scriptscriptstyle lim}{\subset}}
\def\qsim{\overset{\scriptscriptstyle \, q}{\sim}}
\def\qsbs{\overset{\scriptscriptstyle \, q}{\sbs}}
\def\zta{\gz_{\eta}}
\def\mq{2m/(q-1)}
\def\qtop{$\C2q$-fine topology\xspace}
\def\BPq{(B-P)$_q$\xspace}
\def\prtql{$\prt_q$-large\xspace}
\def\quasi{$\C2q$-quasi\xspace}
\maketitle
\begin{abstract}We study the equation $-\Gd u+h(x)|u|^{q-1}u=0$, $q>1$, in $\RN_+=\BBR^{N-1}\ti \BBR_+$
where $h\in C(\ovl{\RN_+})$, $h\geq 0$. Let $(x_1,\ldots, x_N)$ be a coordinate system \sth $\RN_+=[x_N>0]$ and denote a point $x\in \RN$ by $(x',x_N)$. Assume that $h(x', x_N)>0$ when $x'\neq 0$ but  $h(x',x_N)\to 0$ as $|x'|\to 0$. For this class of equations we obtain sharp necessary and sufficient conditions in order that singularities on the boundary do not propagate in the interior.
\end{abstract}
\section{Introduction}
\label{Introduction} \setcounter{equation}{0}
In this paper we
study solutions of the equation
\begin{equation}\label{eqh}
-\Gd u+h(x)|u|^{q-1}u=0,
\end{equation}
in $\RN_+=\BBR^{N-1}\ti \BBR_+$ where $q>1$
and $h\in C(\ovl{\RN_+})$, $h\geq 0$. (If $x\in \RN_+$ we write $x=(x',x_N)$ where $x'=(x_1,\ldots,x_{N-1})$.)

If $h>0$ in $\RN_+$ then boundary singularities of solutions of \req{eqh} cannot propagate to the interior. This is due to the presence of the absorption term  $h|u|^{q-1}u$ and the Keller -- Osserman estimates, \cite{Keller} and \cite{Oss}. In fact, in this case,
\req{eqh} possesses a maximal solution $U$ in $\RN_+$ and,
\begin{equation}\label{Utin}
\lim_{\BSA{c} x_N\to0\\|x|\leq M\ESA}U(x)=\infty \forevery M>0.
\end{equation}
A solution satisfying this boundary condition is called a \emph{large solution}. It is known that under these conditions the large solution is unique (see e.g. \cite{BanM}).

On the other hand, if $h$ vanishes on a set $F\sbs \RN_+$ which has limit points on $[x_N=0]$ then a singularity at these limit points may propagate to the interior. By this we mean that there may exist a \seq $\{u_n\}$ of solutions of \req{eqh} in $\RN_+$ which converges in
$$\Gw=\RN_+\sms F$$
but tends to infinity at some points of $F$.

In this paper we shall study the case where $h$ is positive in $\Gw$ but may vanish on
$$F=\{ (0,x_N)\in \RN_+:\,x_N>0\}.$$

Since $h$ is positive in $\RN_+\sms F$ a singularity at the origin may propagate only along the set $F$.
Furthermore a weak singularity, such as that of the Poisson kernel, cannot propagate to the interior because any solution of \req{eqh} is dominated by the harmonic function with the same boundary behavior. Therefore we must consider only strong singularities, i.e. singularities which cannot occur in the case of a harmonic function but may occur  \wrto solutions of \req{eqh}.

Suppose that
$$h(x',x_N)\leq h_0(|x'|),$$
where
$$h_0\in C^1[0,\infty),\q h_0(s)>0 \text{ for }s>0,\q h_0(0)=0.$$
It is clear that, the faster  $h_0(s)$ tends to zero as $s\to 0$ the greater the chance that a strong boundary singularity at the origin will propagate to the interior.

Our aim is to determine a sharp criterion for the propagation of singularities \wrto solutions of \req{eqh} with $h\in C(\ovl{\RN_+})$
\sth $h>0$ in $\RN_+\sms F$.
It turns out that such a criterion can be expressed in terms of functions of the form
\begin{equation}\label{h-exp}
\bar h(s):= e^{-\frac{\gw(s)}{s}}.
\end{equation}
We assume that $\gw$ satisfies the following conditions:
\begin{equation}\label{tech1}\BAL
(i)\q &\gw\in C(0,\infty)\q\text{is a positive nondecreasing function,}\\
(ii)\q & s\mapsto \mu(s):=\frac{\omega(s)}s\text{ is monotone decreasing on $\BBR_+$},\\
(iii)\q&  \lim_{s\to0}\mu(s)=\infty.
\EAL\end{equation}
We establish the following results.

\begin{theorem}\label{Th.1} Suppose that
\begin{equation}\label{h-1}
\liminf_{\BSA{c} x\to 0\\ x'\neq 0\ESA} h(x)/ \bar h(|x'|)>0
\end{equation}
where $\bar h$ is given by \req{h-exp} and  that \req{tech1} holds.

Suppose that $\gw$ satisfies the Dini condition,
\begin{equation}\label{Dini1}
\int_0^1 (\gw(t)/t)\, dt<\infty.
\end{equation}
If $\{u_n\}$ is a \seq of solutions of \req{eqh} in $\RN_+$ converging (pointwise) in
 $$\Gw=\RN_+\sms F$$
 then the \seq converges in $\RN_+$ and its limit is a solution of \req{eqh} in $\RN_+$.

In particular, \req{eqh} possesses a maximal solution $U$ in $\RN_+$.
\end{theorem}
\begin{theorem}\label{Th.2} Suppose that there exists a constant $c>0$
\sth
\begin{equation}\label{h-2}
h(x)\leq c \,\bar h(|x'|) \forevery x\in \RN_+
\end{equation}
where $\bar h$ is given by \req{h-exp}. Assume that \req{tech1} and the following additional conditions hold:
\begin{equation}\label{tech1'}
 \limsup_{j\tin}\frac{\mu(a^{-j+1})}{\mu(a^{-j})}<1 \q\text{for some $a>1$}
\end{equation}
and
\begin{equation}\label{tech1''}
\lim_{s\to0}\mu(s)/|\ln s|=\infty.
\end{equation}
Condition \req{tech1''} guarantees that, for every real $k$, \req{eqh} has a solution $u_{0,k}$ with boundary data $k\gd_0$ (where $\gd_0$ denotes the Dirac measure at the origin).

Under these assumptions, if
\begin{equation}\label{nonDini}
\int_0^1 (\gw(t)/t)\, dt=\infty
\end{equation}
then
\begin{equation}\label{u0,infty}
u_{0,\infty}=\lim u_{0,k}
\end{equation}
is a solution of \req{eqh} in $\Gw$ but
$$u_{0,\infty}(x)=\infty\forevery x\in F.$$
\end{theorem}

\begin{corollary}\label{cor1}
Suppose that there exists a positive constant $c$ \sth
  \begin{equation}\label{hsim}
    c^{-1}\bar h(|x'|)\leq h(x)\leq c\bar h(|x'|) \forevery x\in \ovl{ \RN_+}
  \end{equation}
  where $\bar h$ is given by \req{h-exp} and satisfies
conditions \req{tech1}, \req{tech1'} and \req{tech1''}. Then the Dini condition \req{Dini1} is necessary and sufficient for the existence of a large solution of \req{eqh} in $\RN_+$. It is also necessary and sufficient for the existence of the strongly singular solution  $u_{0,\infty}$.
\end{corollary}

Problems concerning the propagation of singularities for semilinear equations with absorption  have been studied in \cite{MV-ANS}, \cite{ShV1} (parabolic case) and in \cite{MV-Pisa}, \cite{ShV2} (elliptic case). However,
in  these papers it was assumed that the absorption term is positive everywhere in the interior of the domain, fading only at the initial time  or on the spatial boundary. Consequently singularities could propagate only along $t=0$ or along the boundary.

In \cite{MV-ANS} the authors studied the equation
$$\prt_tu-\Gd u + e^{-\rec{t}}u^q=0\q \text{in }\RN\ti\BBR_+$$
and proved that if $u$ is a positive solution with strong singularity at a point on $t=0$ then $u$ blows up at every point of the initial plane. In \cite{MV-Pisa} the authors studied the corresponding elliptic problem in a domain $D$ where the coefficient of the  absorption term is $e^{-\rec{\gr(x)}}$, $\gr(x)=\dist(x,\bdw)$, proving a similar result.

In \cite{ShV1} the authors considered the equation,
$$\prt_tu-\Gd u + e^{-\frac{\gw(t)}{t}}u^q=0\q \text{in }\RN\ti\BBR_+$$
where $\gw$ is a positive,  continuous and increasing function on $\BBR_+$.
They proved that if $\sqrt \gw$ satisfies the Dini condition then there exist solutions with a strong isolated singularity at a point on $t=0$.
Similar  sufficient conditions were obtained in \cite{ShV2} and \cite{ShV3} \wrto an elliptic (respectively parabolic) equation where the   absorption term vanishes at the boundary (respectively at $x=0$).

The methods of the present paper  can be applied to these and other problems with fading absorption, to obtain sharp necessary and sufficient conditions for the propagation of singularities.
\medskip

\noindent\textbf{Acknowledgment.}
AS wishes to thank the Department of Mathematics at the Technion for its hospitality during his visits.

\section{Proof of Theorem \ref{Th.1}}

Given $R>0$ let $x^R=(0,R)$ and denote by $B_R$ the ball of radius $R$ centered at $x^R$.
We shall prove the following:
  \begin{theorem}\label{Th.1'} Suppose that $h=\bar h$ in a neighborhood of the origin. Then, under the assumptions Theorem \ref{Th.1},  there exists $R>0$ \sth \req{eqh} has a solution $V^R$ in $B_R$ which blows up everywhere on the boundary:
$$V^R(x)\tin \q\text{as}\q x\to\partial B_R.$$
\end{theorem}

Now let $v_k$ denote the solution of \req{eqh} in $\RN_+$ \sth $v_k=k$ on the boundary and put
$$V=\lim_{k\tin} v_k.$$
Condition \req{h-1} implies that there exist positive constants $\unl c$ and $\unl R$ \sth
\begin{equation}\label{h-1'}
 h(x)\geq \unl{c}\,\bar h(|x'|) \q\text{for} \q |x|<\unl R.
\end{equation}
Therefore Theorem \ref{Th.1'} implies that there exists $R\in (0,\unl R/2)$ \sth
$$V\leq V^R.$$
Further this implies that $V$ is locally bounded in the strip $0<x_N<R$ and therefore, everywhere in $\RN_+$. Finally, since $V$ dominates every solution of \req{eqh}, the conclusion of Theorem \ref{Th.1} follows.
\medskip

 The proof of Theorem \ref{Th.1'} is based on estimates of certain energy integrals of solutions of \req{eqh}. In a half space these integrals are infinite.
 Therefore we shall estimate integrals  over a bounded domain for solutions with arbitrary large boundary data.

Condition \req{Dini1} implies that $\lim_{s\to 0}\gw(s)=0$ while \req{tech1} implies that $\lim_{s\to 0}\bar h(s)=0$. We
 extend both of these functions to $[0,\infty)$ by setting them equal to zero at the origin.

 In the course of the proof we denote by $c$, $c'$, $c_i$
constants which depend only on
 $N,q$. The value of the constant may vary from one formula to another. A notation such as $C(b)$ denotes a constant
 depending on the parameter $b$ as well as on $N,q$.

\subsection{Part 1.}
Let $R, b$  be positive numbers \sth $R/8<b<R/2$. Denote by $U_M$, $M>0$,  the solution of \req{eqh} in $B_R(0)$ \sth $U_M=M$ on $\prt B_R(0)$.

  Let
$$\Gw_b=\{x=(x',x_N)\in \BBR^N:|x'|<b,\;|x_N|<b\}.$$
We start with an elementary estimate of the energy integral:
\begin{equation}\label{energy}
    I_b(M)=\int_{\Gw_b}(|\nabla U_M|^2 + h(x)U_M^{q+1})\,dx.
\end{equation}

\begin{lemma}\label{L2.1}
Let $h$ be as in \eqref{h-exp} and assume \req{tech1}. Then
\begin{equation}\label{1.2.2}
I_b(M)\leq C_1(b)M^{q+1},\q C_1(b)=cb^N \bar h(8b).
\end{equation}
\end{lemma}
\begin{proof}
Let $v_M:=U_M-M$.  Multiplying \eqref{eqh} (for $u=U_M$) by $v_M$ and integrating by parts we obtain,
$$\int_{B_R(0)}(|\nabla U_M|^2+h(x)U_M^{q}v_M)dx=0.$$
Therefore
\begin{equation}\label{1.2.5}\BAL
I_b(M)\leq& \int_{B_R(0)}(|\nabla U_M|^2+h(x)U_M^{q+1})dx\\
=&M\int_{B_R(0)}h(x)U_M^qdx \leq c'M^{q+1}\bar h(R)R^N\leq cb^N\bar h(8b)M^{q+1}.
\EAL\end{equation}
\end{proof}
\medskip

\noindent\emph{Notation. } Put
\begin{equation}\label{1.2.7}
\Omega_b(s):=\{x\in\mathbb R^N:s<|x'|<b-s,\
|x_N|<b-s\}\quad\forall\;s\in(0,b/2).
\end{equation}
If $v$ is a positive solution of \req{eqh} in $B_R(0)$, denote
\begin{equation}\label{1.2.8}
J_{b}(s;v):=\int_{\Omega_b(s)}(|\nabla_x v|^2 +\bar h(|x'|)v^{q+1})dx.
\end{equation}
Finally denote,
\begin{equation}\label{varphi}
   \varphi_b(s):=\int_{\partial\Omega_b(s)}h(x)^{-\frac2{q-1}}d\sigma.
\end{equation}

\begin{proposition}\label{L2.2}  There exists a constant $c$ \sth, for every positive solution $v$ of \req{eqh} in $B_R(0)$,
\begin{equation}\label{1.2.9}
J_{b}(s;v)\leq
c\bigg(\int_0^s\varphi_b(r)^{-\frac{q-1}{q+3}}dr\bigg)^{-\frac{q+3}{q-1}}\quad\forall\,s\in (0,b/2).
\end{equation}
\end{proposition}
\begin{proof} Put $S_b(s):=\partial\Omega_b(s)$ and denote by $\vec{n}=\vec{n}(x)$ the unit outward normal  to $S_b(s)$
 at $x$.

Multiplying equation \req{eqh}  by $v$ and integrating by parts over $\Gw_b(s)$ we obtain,
\begin{equation}\label{1.2.10}
\int_{\Omega_b(s)}(|\nabla_xv|^2+\bar h(|x'|)v^{q+1})dx=\int_{S_b(s)}\frac{\partial
v}{\partial \vec{n}}vd\sigma,
\end{equation}

We estimate the term on the  right-hand
side  using first H\"older's inequality (for a product of three terms) and secondly Young's inequality:
\begin{equation}\label{1.2.11}\BAL
\bigg|&\int_{S_b(s)}v\frac{\partial v}{\partial \vec
n}d\sigma\bigg| \leq
\int_{S_b(s)}|\nabla_xv||v|d\sigma\leq\\
\bigg(&\int_{S_b(s)}|\nabla_xv|^2d\sigma\bigg)^{\frac{1}{2}}\bigg(\int_{S_b(s)}h(x)
|v|^{q+1}d\sigma\bigg)^{\frac{1}{q+1}}
\varphi_b(s)^{\frac{q-1}{2(q+1)}}\leq\\
c_1\bigg(&\int_{S_b(s)}(|\nabla_xv|^2+h(x)v^{q+1})d\sigma\bigg)^{\frac{q+3}{2(q+1)}}
\varphi_b(s)^{\frac{q-1}{2(q+1)}}.
\EAL\end{equation}
Substituting estimate \eqref{1.2.11} into
\eqref{1.2.10} we obtain:
\begin{multline}\label{1.2.13}
J_{b}(s;v)\leq
c_2\bigg(\int_{S_b(s)}(|\nabla_xv|^2+h(x)v^{q+1})d\sigma\bigg)^{\frac{q+3}{2(q+1)}}
\varphi_b(s)^{\frac{q-1}{2(q+1)}}.
\end{multline}
Since
$$
-\frac{d}
{ds}J_{b}(s;v)=\int_{S_b(s)}(|\nabla_xv|^2+h(x)v^{q+1})d\sigma,
$$
inequality \eqref{1.2.13}  is equivalent to
$$
J_{b}(s;v)\leq c_3\varphi_b(s)^{\frac{q-1}{2(q+1)}}\Big(-\frac
d{ds}J_{b}(s;v)\Big)^{\frac{q+3}{2(q+1)}}\qquad\forall\, s\in (0,b/2).
$$
 Solving this differential
inequality, with initial data $J_b(b/2;v)=0$, we obtain \eqref{1.2.9}.
\end{proof}
\medskip

In continuation we derive a more explicit estimate for $h$
as in  \eqref{h-exp}.
 We need the following technical lemma.
 \begin{lemma}\label{L2.3}
Let $A>0$, $m\in\mathbb N$, $ l\in\mathbb R^1$ and let $\omega\in C^1(0,\infty)$ be a positive
function satisfying condition \eqref{tech1}. Then there exist
$\bar s\in (0,1)$, depending on $A,l$ and $\gw$ such that the following inequality holds:
\begin{multline}\label{1.2.19}
\int_0^s
t^{m-1}\omega(t)^l\exp\Big(-A\mu(t)\Big)dt\geq\frac{s^{m+1}\omega(s)^{l-1}}{(m+1)\mu(s)^{-1}+A}
\exp(-A\mu(s))\\\forall s:0<s<\bar s.
\end{multline}
\end{lemma}
\begin{proof}
Due to condition \eqref{tech1} (ii)  integration by parts yields:
\begin{equation}\label{1.2.20}\BAL
&\int_0^st^m\omega(t)^l\exp(-A\mu(t))dt\\
=&\frac{s^{m+1}}{m+1}\omega(s)^l\exp(-A\mu(s))-
\int_0^s\frac{At^{m-1}}{m+1}\exp(-A\mu(t))\omega(t)^{l+1}dt\\
+&\int_0^s\frac{t^{m+1}}{m+1}\exp(-A\mu(t))\omega'(t)\omega^{l-1}( A\mu(t)-l)dt.
\EAL\end{equation}
Again due to \eqref{tech1} (ii), there exists $\bar s>0$ such that
$$A\mu(s)\geq l\quad\forall s\in(0,\bar s).$$
For later estimates it is convenient to choose $\bar s$ in $ (0,1)$.

As $\omega(s)$ is non-decreasing, it follows that, for $0<s\le \bar s$,
\begin{multline*}
\Big(s+\frac{A\omega(s)}{m+1}\Big)\int_0^st^{m-1}\omega(t)^l\exp(-A\mu(t))dt
\geq\frac{s^{m+1}}{m+1}\omega(s)^l\exp(-A\mu(s)).
\end{multline*}
This inequality is equivalent to \eqref{1.2.19}.
\end{proof}



\begin{proposition}\label{J0j}
 Assume  that $h$ is given by  \eqref{h-exp}  and  satisfies \eqref{tech1}. Then there exists a constant $s^*\in (0,b/2)$, depending on $N,q$ and the rate of blow-up of $\mu(s)=\gw(s)/s$ as $s\to 0$, \sth
\begin{equation}\label{1.2.22}
\begin{gathered}
J_b(s;v)\leq cb^{N-1}\exp Q(s) \forevery s\in (0,s^*)\\
Q(s)=\frac{2\mu(s)}{q-1}+\frac{q+3}{q-1}\ln\mu(s)-\frac{q+3}{q-1}\ln s,
\end{gathered}
\end{equation}
for every positive solution $v$ of \req{eqh} in $B_R(0)$.

If, in addition, there exists a positive constant $\gb$ such that
\begin{equation}\label{tech2}
\gb \ln \rec{s}\leq {\mu(s)}\quad 0<s\le s^*,
\end{equation}
 then
\begin{equation}\label{estJj}
Q(s)\leq Q_0\mu(s) \quad 0<s\le s^*
\end{equation}
where
\begin{equation}\label{estP0}
Q_0:=\frac2{q-1}+\frac{q+3}{(q-1)}+\frac{q+3}{\gb(q-1)}.
\end{equation}

\end{proposition}
\begin{proof} Denote
$$S_{b,1}(s)=\{x:|x'|=s,\;|x_N|<b\}\cup\{x:|x'|=b-s,\; |x_N|<b\}$$
and
$$S_{b,2}(s)=\{x: s<|x'|<b-s,\;|x_N|=b\}.$$
Then
\begin{equation}\label{1.2.14}\BAL
&\int_{S_{b,1}}\bar
h(|x'|)^{-\frac2{q-1}}d\sigma\\&=2\gg_{N-1}(b-s)(\bar
h(s)^{-\frac2{q-1}}s^{N-2}+\bar
h(b)^{-\frac2{q-1}}(b-s)^{N-2})\\
&\leq  4b^{N-1}\gg_{N-1}\exp\frac{2\mu(s)}{q-1}\qquad
0<s<b/2,
\EAL\end{equation}
where $\gg_{N-1}$ denotes the area of the unit sphere in $\mathbb
R^{N-1}$.
Further, since $\mu$ is monotone decreasing,
\begin{multline}\label{1.2.16}
\int_{S_{b,2}}\bar h(|x'|)^{-\frac2{q-1}}d\sigma=
2\gg_{N-1}\int_{s}^{b-s}\exp\frac{2\mu(\rho)}{q-1}
\rho^{N-2}d\rho\\
\leq
2(N-1)^{-1}b^{N-1}\gg_{N-1}\exp\frac{2\mu(s)}{q-1}.
\end{multline}
By    \eqref{1.2.14} and \eqref{1.2.16}:
$$
\varphi_b(s)=\int_{S_b(s)}\bar h(|x'|)^{-\frac2{q-1}}d\sigma\leq
cb^{N-1}\exp\frac{2\mu(s)}{q-1}, \q 0<s<b/2,
$$
where $c=(4+2(N-1)^{-1})\gg_{N-1}$. This implies,
\begin{equation}\label{1.2.17}
\int_0^s\varphi_b(r)^{-\frac{q-1}{q+3}}dr\geq
c_1b^{-\frac{(N-1)(q-1)}{q+3}}\int_0^s\exp\Big(
-\frac{2\mu(r)}{q+3}\Big)dr,\quad c_1=c^{-\frac{q-1}{q+3}}.
\end{equation}

Let $s^*$ be the largest number in $(0,b/2)$ \sth
\begin{equation*}\BAL
&\diamond\q \text{$s^*\le\bar s$, ($\bar
s$ as in Lemma \ref{L2.3} for  $l=0$, $m=1$ and
$A=\frac{2}{q+3}$),}\\
&\diamond\q \text{$\mu(s^*)\geq A^{-1}=(q+3)/2$.}
\EAL\end{equation*}
Then \req{1.2.17} and  \req{1.2.19} imply
\begin{equation}\label{1.2.21}\BAL
\int_0^s\varphi_b(r)^{-\frac{q-1}{q+3}}dr \geq\;
&c_2b^{-\frac{(N-1)(q-1)}{q+3}}\frac{s^2}{\omega(s)}\exp\Big(-\frac{2\mu(s)}{q+3}\Big),\\
c_2=\;&c_1(q+3)/6, \EAL
\end{equation}
for all $s\in (0,s^*]$.
This inequality and  \eqref{1.2.9} imply \eqref{1.2.22}.
\medskip

Suppose now that the function $\mu(\cdot)$ given by \eqref{h-exp} satisfies \eqref{tech2}.
Since $\ln r\leq r $ for $r\geq 1$,
conditions \eqref{tech1}, \eqref{1.2.22} and \eqref{tech2} imply  \req{estJj}.
\end{proof}
\medskip

Next we  estimate energy integrals over domains of the form
\begin{equation}\label{1.2.23}
\Gw_b(\tau,\sigma):=\{x=(x',x_N):|x'|<\sigma,\;
|x_N|<b-\tau\}
\end{equation}
where $0<\sigma<b/2$, $0\leq\tau<b$.

Let $\eta\in C^\infty([0,\infty))$ be a monotone decreasing function such that

\begin{equation}\label{1.2.24}
\eta(s)=1\text{ if }s<1,\ \eta(s)=0\text{ if }s>2,\ \eta'(s)\leq 2
\end{equation}
and denote
$$\eta_\gs(s)=\eta(s/\gs).$$
We shall estimate the integrals,
\begin{equation}\label{1.2.25}
E_{b}(\tau,\sigma;v):=
\int_{\Gw_b(\tau,2\sigma)}\Big(\,\big|\nabla_x\big(\eta_\sigma(|x'|)
v\big)\big|^2+h(x)\eta_\sigma(|x'|)^2v^{q+1}\Big)\,dx.
\end{equation}
\begin{proposition}\label{L2.4}
Assume condition \eqref{tech1}. Let $s^*\in (0,b/2)$ be as in Proposition \ref{J0j}. Then the following inequality holds for   $0<\gs\le s^*$ and $\gs\le\tau<b$:
\begin{equation}\label{1.2.26}
E_b(\tau,\sigma;v)\leq
c\,\sigma \Big(-\frac{dE_b(\tau,\sigma;v)}{d\tau}\Big)+C_2(b)\exp
H(\sigma),
\end{equation}
 where $C_2(b):=cb^{\frac{2(N-1)}{q+1}}$,
\begin{equation}\label{Hsigma}\BAL
&H(\sigma)=2\frac{Q(\sigma)+\mu(\sigma)}{q+1} +\frac{(N-1)(q-1)-2(q+1)}{q+1}\ln\sigma\\
&=\frac{2\mu(\sigma)}{q-1}+\frac{2(q+3)}{q^2-1}\ln\mu(\gs)
-c^*_+\ln \sigma \EAL\end{equation} and
$$c^*=\frac{2(q+3)+2(q^2-1)-(N-1)(q-1)^2}{q^2-1}.$$

 If, in
addition, condition \eqref{tech2} holds then there exists a
constant $H_0$ depending only on $q$ and $\gb$ \sth
\begin{equation}\label{estEj}
H(\gs)\le H_0\mu(\gs),
\end{equation}
where
\begin{equation}\label{H0}
H_0=\frac2{q-1}+\frac{2(q+3)}{(q-1)(q+1)}+\frac{c^*_+}{\gb}.
\end{equation}
\end{proposition}
\begin{proof}
Multiplying equation \req{eqh}  by $\eta_\gs(|x'|)^2 v$ and
integrating by parts over $\Gw_b(\tau,2\gs)$ we obtain,
\begin{equation}\label{1.2.27}\BAL
\int_{\Gw_b(\tau,2\gs)}\nabla
v\cdot \nabla(v\eta_\sigma^2)dx+&\int_{\Gw_b(\tau,2\gs)}h(x)v^{q+1}\eta_\sigma^2dx\\
=& \int_{S'_b(\tau,2\gs)}\frac{\partial v}{\partial \vec
n}v\eta_\sigma^2dx', \EAL\end{equation} where
$S'_b(\tau,\gs)=\{x:|x'|<\sigma,\ |x_N|=b-\tau\}$.

We estimate the first term on the left hand side:
\begin{equation}\label{1.2.28}\BAL
&\int_{\Gw_b(\tau,2\gs)}\nabla
v\cdot \nabla(v\eta_\sigma^2)dx=\\
&\int_{\Gw_b(\tau,2\gs)}|\nabla(v\eta_\sigma)|^2dx-\int_{\Gw_b(\tau,2\gs)}v^2\absol{\nabla\eta_\gs}^2dx\geq\\
&\int_{\Gw_b(\tau,2\gs)}|\nabla(v\eta_\sigma)|^2dx -
4\sigma^{-2}\int_{\tilde \Gw_b(\tau,\gs)}v^2dx, \EAL\end{equation}
where
\begin{equation}\label{tlGw}
\tilde\Gw_b(\tau,\gs):=\{\sigma<|x'|<2\sigma,\ |x_N|<b-\tau\}.
\end{equation}
Using H\"{o}lder's inequality,  conditions \eqref{h-exp},
\eqref{tech1} and estimate \eqref{1.2.22} with $s=\gs$, we obtain:
\begin{equation}\label{1.2.30}\BAL
&\int_{\tilde \Gw_b(\tau,\gs)}v(x)^2dx\leq\\
\Big(&\int_{\tilde
\Gw_b(\tau,\gs)}v^{q+1}h(x)dx\Big)^{\frac2{q+1}}\Big(\int_{\tilde
\Gw_b(\tau,\gs)}h(x)^{-\frac2{q-1}}dx\Big)^{\frac{q-1}{q+1}}\leq\\
&
c'(b^{N-1}\exp Q(\sigma))^{\frac2{q+1}}\bar
h(\sigma)^{-\frac2{q+1}}\absol{\tilde
\Gw_b(\tau,\gs)}^{\frac{q-1}{q+1}}\leq\\
& cb^{\frac{2(N-1)}{q+1}}
\exp\Big(\frac{2Q(\gs)}{q+1}\Big)\exp\Big(\frac{2\mu(\gs)}{q+1}\Big)
\sigma^{\frac{(N-1)(q-1)}{q+1}} \EAL\end{equation} for
$\gs<\tau<b$ and $0<\gs<\min\{s^{\ast},\frac{b}{3}\}$. The
application of \eqref{1.2.22} here is justified because, for
$\tau$ and $\gs$ as above, $\tilde \Gw_b(\tau,\gs)\subset
\Gw_b(\gs)$.

Combining \req{1.2.27}  -- \req{1.2.30} we obtain,
\begin{equation}\label{1.2.29}\BAL
&\int_{\Gw_b(\tau,2\gs)}|\nabla(v\eta_\sigma)|^2dx+\int_{\Gw_b(\tau,2\gs)}h(x)v^{q+1}\eta_\sigma^2dx\leq\\
&\int_{S'_b(\tau,2\gs)}\frac{\partial v}{\partial \vec
n}v\eta_\sigma^2dx' + cb^{\frac{2(N-1)}{q+1}}
\exp\Big(\frac{2(Q(\gs)+\mu(\gs)}{q+1}\Big)
\sigma^{\frac{(N-1)(q-1)}{q+1}-2}, \EAL\end{equation} Next, by
H\"{o}lder's inequality,
\begin{multline*}\label{1.2.31'}
\bigg|\int_{S'_b(\tau,2\gs)}\frac{\partial v}{\partial\vec
n}v\eta_\sigma^2dx'\bigg|\leq
\int_{S'_b(\tau,2\gs)}\Big|\frac\partial{\partial
x_N}(v\eta_\sigma(|x'|))\Big|v\eta_\sigma\,dx'\\
\leq \bigg(\int_{S'_b(\tau,2\gs)}\big(\frac\partial{\partial
x_N}(v\eta_\sigma)\big)^2dx'\bigg)^{1/2}\bigg(\int_{S'_b(\tau,2\gs)}
(v\eta_\sigma)^2dx'\bigg)^{1/2}
\end{multline*}
and by  Poincar\'e's inequality in $S'_b(\tau,\gs)$,
$$\int_{S'_b(\tau,2\gs)}(v\eta_\sigma)^2dx'\leq (c_0\sigma)^2\int_{S'_b(\tau,2\gs)}
|\nabla_{x'}(v\eta_\sigma)|^2dx'.$$
Therefore
\begin{equation}\label{1.2.31}
\bigg|\int_{S'_b(\tau,2\gs)}\frac{\partial v}{\partial\vec
n}v\eta_\sigma^2dx'\bigg|\leq c\sigma\int_{S'_b(\tau,2\gs)}
|\nabla_{x}(v\eta_\sigma)|^2dx'.
\end{equation}
Since
$$
\frac{dE_b(\tau,\gs;v)}{d\tau}=
-\int_{S'_b(\tau,2\gs)}(|\nabla(v\eta_\sigma)|^2+h(x)v^{q+1}\eta_\sigma^2)dx'.
$$
inequalities
\eqref{1.2.29} and \eqref{1.2.31}  imply \eqref{1.2.26}.

Finally,  if \eqref{tech2} holds, \eqref{estEj} is obtained in the same way as
 \req{estJj}.
\end{proof}

\subsection{Part2.}

\noindent\emph{Notation.} Given $M>0$ and $\nu\in (0,1)$, let $s_\nu=s_\nu(M)$ be defined by,
\begin{equation}\label{shift}
   \exp(Q_0\mu(s_\nu(M))=\bar h(s_\nu(M))^{-Q_0}=M^\nu,
\end{equation}
where $Q_0$ is given by \req{estP0}.

\begin{lemma}\label{L-barEM}
Put
\begin{equation}\label{1.2.37}
\gamma:=\frac{2(q+1+\beta)-(N-1)(q-1)}{\beta Q_0(q+1)},
\end{equation}
where $\beta$ is a positive number satisfying \eqref{tech2} and
\begin{equation}\label{1.2.37'}
\nu_0:=\begin{cases} 1 \q&\text{if }\gamma\leq 0,\\
\frac{q-1}{\gamma} \q &\text{if }\gamma>0.
\end{cases}
\end{equation}
If
\begin{equation}\label{nu<nu0}
0<\nu<\min(\nu_0,1)
\end{equation}
then,
 \begin{equation}\label{1.2.38}
  E_b(0,s_\nu(M');U_M)\leq  2\big(  I_b(M)+ C_3(b) M^2M'^{q-1}\big)\q 1\le M'\le
  M,
\end{equation}
where
\begin{equation}\label{C_3}
   C_3(b):=c b^{\frac{2N+q-1}{q+1}}\bar
h(8b)^{\frac{2}{q+1}}.
\end{equation}
\end{lemma}

\begin{proof} Put
$$ I'_b(s,M):=\int_{\Gw_b} U_M^2|\nabla\eta_{s}|^2dx. $$

Then,
 \begin{equation}\label{E<I+I'}
 \BAL
 &E_b(0,s_\nu(M'),U_M)\\
 &\le 2\int_{\Gw_b}(|\nabla(U_M)|^2\eta_{s_\nu}^2+h(x)U_M^{q+1}\eta_{s_\nu}^2)dx+
2\int_{\Gw_b} U_M^2|\nabla\eta_{s_\nu}|^2dx\\
&\le 2\big(I_b(M) + I'_b(s_\nu,M)\big), \q s_\nu=s_\nu(M').
\EAL\end{equation}

By \req{1.2.24}, $\nabla\eta_{s_\nu}(|x'|)=0$ for $|x'|<s_\nu$ and
for $|x'|>2s_\nu$. Therefore, applying H\"older's inequality and
using  the monotonicity of $\bar h$ we obtain
 $$ \BAL
  I'_b(s_\nu(M'),M)&\leq 4s_\nu^{-2}\int_{\tl\Gw_{b}(0,s_\nu)}U_M^2dx \\
 &\leq 4s_\nu^{-2}\Big(\int_{\tl\Gw_{b}(0,s_\nu)} U_M^{q+1}h\,dx\Big)^{\frac{2}{q+1}}
 \Big(\int_{\tl\Gw_{b}(0,s_\nu)} \bar h(|x'|)^{\frac{2}{1-q}}dx\Big)^{\frac{q-1}{q+1}}\\
 &\leq cs_\nu^{-2}(b^N\bar h(8b) M^{q+1})^{\frac{2}{q+1}}
 \bar h(s_\nu)^{-\frac{2}{q+1}}s_\nu^{\frac{(N-1)(q-1)}{q+1}}b^{\frac{q-1}{q+1}}\\
 &=c(b^N\bar h(8b) )^{\frac{2}{q+1}}b^{\frac{q-1}{q+1}} M^2 s_\nu^{-2+\frac{(N-1)(q-1)}{q+1}}\exp{\frac{2\mu(s_\nu)}{q+1}}.
 \EAL$$
By \req{tech2}  and  \req{shift}
$$s^{-1}\leq \exp{(\mu(s)/\gb)},\q M'^{-\nu/Q_0}=\bar h(s_\nu)=\exp{(-\mu(s_\nu))}.$$
 Therefore the previous inequality yields
 $$  I'_b(s_\nu(M'),M)\leq  c(b^N\bar h(8b) )^{\frac{2}{q+1}}b^{\frac{q-1}{q+1}}M^2M'^{\frac{\nu}{\gb Q_0}
 (2-\frac{(N-1)(q-1)}{q+1})+\frac{2}{q+1}\frac{\nu}{ Q_0}}.$$
Hence
\begin{equation}\label{1.2.36}
 I'_b(s_\nu(M'),M)\leq C_3(b)M^2M'^{\nu\gg}
\end{equation}
with $\gg$ and $C_3(b)$ as in \eqref{1.2.37} and \eqref{C_3}.
By \eqref{nu<nu0} $\nu\gg\le q-1$. Therefore \req{E<I+I'} and \req{1.2.36} imply
\req{1.2.38}.
\end{proof}
\medskip

\noindent\emph{Notation. } For every $M>0$ and $0\leq s\leq b/2$ denote,
\begin{equation}\label{TnuM}
  T_b(s,M)=\{\tau:\,s\leq \tau< b,\;  E_b(\tau,s;U_M)\geq 2C_2(b)\exp(H_0\mu(s))\}
\end{equation}
where $C_2(b)$ is the constant in \req{1.2.26} and  $H_0$ is given by \req{H0}.

Note that $\tau\mapsto E_b(\tau,s;U_M)$ is continuous and non-increasing in the interval $[s,b]$.
Therefore, if
$$E_b(s,s;U_M)<2C_2(b)\exp(H_0\mu(s))$$
then $T_b(s,M)=\ems$. Put,
\begin{equation}\label{tausM}
\tau_b(s,M)=\begin{cases} s\q &\text{if }T_b(s,M)=\ems,\\
\sup T_b(s,M) \q&\text{otherwise}
\end{cases}
\end{equation}
and
\begin{equation}\label{taunu}
\tau_{b,\nu}(M',M):=\tau_b(s_\nu(M'),M).
\end{equation}
Since $\lim_{\tau\to b}E_b(\tau,s;U_M)\to 0$  it follows that
\begin{equation}\label{taunu<b}
  s_\nu(M')\leq \tau_{b,\nu}(M',M)<b.
\end{equation}
Furthermore,
\begin{equation}\label{E_M<}
  E_b(\tau_{b,\nu}(M',M), s_\nu(M');U_M)\leq 2C_2(b)\exp(H_0\mu(s_\nu(M')))
\end{equation}
and,  if $\tau_{b,\nu}(M',M)>s_\nu(M')$ then,
\begin{equation}\label{E_M>}
 E_b(\tau, s_\nu(M');U_M)\geq 2C_2(b)\exp(H_0\mu(s_\nu(M')))
 \end{equation}
for every $\tau\in (0,\tau_{b,\nu}(M',M)]$,  with equality for $\tau=\tau_{b,\nu}(M',M)$.
\medskip

\begin{proposition}\label{IMtau} $(i)$ Let
$$b'_\nu(M',M):={b-\tau_{b,\nu}(M',M)}.$$
 Then
\begin{equation}\label{1.2.45}
\int_{\Gw_{b'_{\nu}(M',M)}}(|\nabla_xU_M|^2+h(x)U_M^{q+1})dx\leq
c_0(b^{N-1}M'^\nu+C_2(b)M'^{\frac{\nu
H_0}{Q_0}}).
\end{equation}

\noindent$(ii)$
Assume that
 \begin{equation}\label{nu2}
0<\nu\leq\frac{q+1}{4} \min(1,Q_0/H_0).
\end{equation}
where $H_0$ is given by \req{H0} and $Q_0$ is given by \req{estP0}. Let $a\in (1,2)$ and
assume that $M'$ is large enough so that,
\begin{equation}\label{Mcon2}
C_4(b):= c_0(b^{N-1}+C_2(b))/ \,C_1(b)\leq M'^{(q+1)/2a}
\end{equation}
where $C_1(b)$ and $C_2(b)$ are the constants in Lemma \ref{L2.1} and Proposition \ref{L2.4} respectively while $c_0$ is the constant in \req{1.2.45}.

Then
\begin{equation}\label{mainstep}
 I_{b'_\nu(M',M)} (M)=  \int_{\Gw_{b'_{\nu}(M',M)}}(|\nabla_xU_M|^2+h(x)U_M^{q+1})dx\leq C_1(b)M'^{\frac{q+1}{a}}.
\end{equation}

\end{proposition}
\begin{proof} By \req{shift},
\begin{equation}\label{snuM}
  M'=\exp\Big({\frac{Q_0}{\nu}\mu(s_\nu(M'))}\Big).
\end{equation}
Therefore, by \req{E_M<},
\begin{equation}\label{1.2.44}
E_b(\tau_{b,\nu}(M',M),s_\nu(M');U_M)\leq 2C_2(b)M'^{\frac{\nu
H_0}{Q_0}}.
\end{equation}
By Proposition \ref{J0j} applied to the estimate of $J_b(s_\nu(M'),U_M)$, 
\begin{equation}\label{1.2.44+}
J_b(s_\nu(M'),U_M)\leq cb^{N-1}\exp(Q_0\mu(s_\nu(M'))= cb^{N-1}M'^\nu.
\end{equation}
Inequality \req{taunu<b} implies that $b'_\nu(M',M)\leq b-s_\nu(M')$. Therefore
$$\Gw_{b'_\nu(M',M)}\subset \Gw_b(\tau_{b,\nu}(M',M),s_\nu(M'))\cup \overline{\Gw_b(s_\nu(M')})$$
(see \req{1.2.7} for definition of $\Gw_b(s)$). \Consy
$$
 I_{b'_\nu(M',M)} (M)\leq  E_b(\tau_{b,\nu}(M',M),s_\nu(M');U_M)+ J_b(s_\nu(M'),U_M).
$$
This inequality together with \req{1.2.44} and \req{1.2.44+}  imply \req{1.2.45}.

In view of \req{nu2} we have,
$$b^{N-1}M'^\nu+C_2(b)M'^{\frac{\nu
H_0}{Q_0}}\leq (b^{N-1}+C_2(b))M'^{(q+1)/2a}.$$
If $M'$ satisfies  \req{Mcon2}, this inequality and
 \req{1.2.45} imply \req{mainstep}.
\end{proof}

Next we derive an upper bound for $\tau_{b,\nu}(M',M)$ in terms of $s_\nu(M')$.
\begin{lemma}\label{tau1}
Suppose that  $0<\nu$ satisfies conditions \req{nu<nu0} and
\req{nu2} and that
\begin{equation}\label{Mcon1}
 M\geq \exp\Big({\frac{Q_0}{\nu}\mu(s^*)}\Big)
\end{equation}
where $s^*$ is as in Proposition \ref{L2.4}. Then
\begin{equation}\label{c_for_k}
\exp\Big(\frac{\tau_{b,\nu}(M',M)}{2cs_\nu(M')}\Big)\leq
c_1(I_b(M)+C_3(b)M^2M'^{q-1})C_2(b)^{-1} M'^{-\frac{\nu H_0}{Q_0}}.
\end{equation}
\end{lemma}
\begin{proof}
 Since $\nu$ satisfies \req{nu2} and $1<a<2$,
$$0<Q_0(q+1)(1-\rec{2a})\leq Q_0(q+1)-H_0\nu.$$
By \eqref{1.2.38},
\begin{equation}\label{incon}\BAL
&E_b(\tau,s_\nu(M');M)\leq\\
 &E_b(0,s_\nu(M');M)\leq 2\big(  I_b(M)+ C_3(b) M^2M'^{q-1}\big) \q \forevery \tau\in (0,b)
\EAL\end{equation}
where $1<M'<M$.


 If $\tau_{b,\nu}\leqslant s_\nu$ inequality \eqref{c_for_k} is
 trivial. Therefore we may assume that
$$\tau_{b,\nu}(M',M)>s_\nu(M').$$
Temporarily denote
$$F(\tau)=E_b(\tau, s_\nu(M');M).$$
 By Proposition \ref{L2.4}, \req{Mcon1} and \req{E_M>},
\begin{equation}\label{hommEj}
F(\tau)\leq 2cs_\nu(M')\Big(-\frac{dF(\tau)}{d\tau}\Big)\quad\forall \tau:s_\nu(M')<\tau<\tau_{b,\nu}(M',M).
\end{equation}
Solving  this differential inequality with initial condition $F(s_\nu(M'))$ satisfying \req{incon} we obtain,
\begin{equation}\label{expEj}\BAL
E_b(\tau, s_\nu(M');M)\leq
c_1(I_b(M)+C_3(b)M^2M'^{q-1})\exp\Big(-\frac{\tau}{2cs_\nu}\Big)
\EAL\end{equation}
for every $\tau\in[s_\nu(M'),\tau_{b,\nu}(M',M)]$.
Combining \eqref{expEj} and  \eqref{E_M>} for $\tau=\tau_{b,\nu}(M',M)$ (in which case \eqref{E_M>} holds with equality) we obtain,
$$\BAL
&2C_2(b)\exp(H_0\mu(s_\nu(M')))\\
&\leq
c_1(I_b(M)+C_3(b)M^2M'^{q-1})\exp\Big(-\frac{\tau_{b,\nu}(M',M)}{2cs_\nu(M')}\Big).
\EAL$$
In view of \req{snuM} this inequality implies
\begin{equation}\label{c_k}\BAL
&\exp\Big(\frac{\tau_{b,\nu}(M',M)}{2cs_\nu(M')}\Big)\leq \\
&c_1(I_b(M)+C_3(b)M^2M'^{q-1})C_2(b)^{-1}\exp(-H_0\mu(s_\nu(M')))=\\
&c_1(I_b(M)+C_3(b)M^2M'^{q-1})C_2(b)^{-1} M'^{-\frac{\nu H_0}{Q_0}}.
\EAL\end{equation}

\end{proof}

\subsection{Part 3.} In this part of the proof we apply the previous estimates to a specific \seq $\{M_j\}$
defined below. As before $R$ is an arbitrary positive number and
we require that $R/4<b<R/2$.
\begin{proposition}\label{main}
Let
\begin{equation}\label{M_j}
   M_j=\exp(a^j), \q s_{j}:=s_\nu(M_j)
\end{equation}
where $s_\nu(\cdot)$ is  defined as in \req{shift} and
\begin{equation}\label{a=}
1<a<\min(1+\frac{\nu H_0}{2Q_0}, 2).
\end{equation}
 Put $u_j=U_{M_j}$.
Then there exists  $j_0\in \BBN$ \sth
\begin{equation}\label{I_b/2}
 \int_{\Gw_{b/2}}(|\nabla_xu_j|^2+h(x)u_j^{q+1})dx\leq   C_1(b)M_{j_0}^{q+1} \forevery j> j_0
\end{equation}
where $C_1(b)=cb^N\bar h(8b)$.
\end{proposition}
\begin{proof}
By \req{M_j} and \req{shift},
\begin{equation}\label{snuj}
a^j\nu/Q_0=\mu(s_{j}).
\end{equation}
Let $j_0$ be a positive integer to be determined later on. For each integer $j\geq j_0$ we define the set of pairs $$\{b_{i,j},\,\tau^{i,j}:\, i=j_0,\ldots,j\}$$
 by  induction as follows:
  $$\BAL
&  \tau^{j,j}=\tau_{b,\nu}(M_j,M_j), \q b_{j,j}=b-\tau^{j,j},\\
&\tau^{i,j}=\tau_{b_{i+1,j},\nu}(M_{i}, M_j),\q b_{i,j}=b_{i+1,j}-\tau^{i,j}, \q j_0\le  i< j.
\EAL$$
Thus
$$b_{i,j}=b-\sum_{k=i}^j \tau^{k,j}, \q j_0\le i< j.$$

We show below that if $j_0$ is sufficiently large then
\begin{equation}\label{<b/2}
 \sum_{i=j_0}^j \tau^{i,j} < b/2\qquad \forall\,j>j_0,
\end{equation}
which implies,
$$b/2<b_{i,j}.$$

Specifically we choose $j_0$  so that,
\begin{equation}\label{Mj0}\BAL
(i)\q & C_4(b/2)\leq M_{j_0}^{(q+1)/2a}\\
(ii)\q &\exp\Big(\frac{Q_0}{\nu}\mu(s^*)\Big)\leq M_{j_0},\\
(iii)\q &C_5(b):=c_1\frac{C_1(b)+C_3(b)}{C_2(b)}\leq  M_{j_0}^{q+1}
\EAL\end{equation}
with  $c_1$ as in \req{c_for_k}. For the definition of $C_1(b),\ldots, C_4(b)$ see  \req{1.2.2}, \req{1.2.26}, \req{C_3} and \req{Mcon2}.

\bcom@
\begin{equation}\label{Mj-con}\BAL
&(i)   &&C_4(b/2)\leq M_{j_0}^{(q+1)/2a},\\
&\;&& C_4(b):= c_0(b^{N-1}+C_2(b))/ \,C_1(b) \\
&(ii) && \exp\Big(\frac{Q_0}{\nu}\mu(s^*)\Big)\leq M_{j_0},\\
& (iii) && c_1\frac{C_1(b)+C_3(b)}{C_2(b)}\leq (q+1)\ln M_{j_0}
\EAL\end{equation}
with $c_0$ as in \req{1.2.45} and $c_1$ as in \req{c_for_k}.
@\end{comment}

We observe that $C_4(b)$ decreases as $b$ increases. Therefore
(assuming \req{<b/2}) condition (i) implies,
\begin{equation}\label{Mj1}
  C_4(b_{i,j})\leq M_{i}^{(q+1)/2a}, \q j_0\le i\le j,\q j_0\le j.
\end{equation}
The left hand side in condition \req{Mj0}(iii) increases as $b$ increases. Therefore
\begin{equation}\label{Mj3}
C_5(b_{i,j})\leq (q+1)\ln M_{i}, \q j_0\le i\le j,\q j_0\le j.
\end{equation}

\bcom @
for $j\geq j_0$,  $M_j$ satisfies conditions \req{Mcon1},  \req{Mcon2}, \req{Mcon3} when $b$ is replaced by $b_{i,j}$, $i=j_0,\cdots,j$.

First we require that
\begin{equation}\label{M2'}
   \kappa_0(b/2)\leq M_{j_0},
\end{equation}
for $\kappa_0$ as in \req{Mcon2}.
Since $\kappa_0$ is monotone decreasing \req{M2'} implies that \req{Mcon2} holds for $b_{i,j}$ if it is larger then $b/2$.

Secondly we require
\begin{equation}\label{M1,3}
a^{j_0}\nu/Q_0\geq \mu(s^*),\q a^{j_0}\geq  \frac{Q_0\ln \bar \kappa(b)}{Q_0(q+1)-H_0\nu}.
\end{equation}
Note that $\bar \kappa$ is monotone increasing so that the second inequality remains valid if $b$ is replaced by $b_{i,j}$ while the first inequality is independent of $b$.
@\end{comment}

Put $u_j=U_{M_j}$. Assuming that \req{<b/2} holds, we apply  Proposition \ref{IMtau} to the case where $b$ is replaced by  $b_{j_0+1,j}$ and  $M'=M_{j_0+1}$, $M=M_j$; we obtain,
\begin{equation}\label{Mjest}
  \int_{\Gw_{b_{j_0,j}}}(|\nabla_xu_j|^2+h(x)u_j^{q+1})dx\leq   C_1(b)M_{j_0}^{q+1}
\end{equation}
which implies \req{I_b/2}.

It remains to verify \req{<b/2}. To this end we prove the following estimate:
\begin{equation}\label{tauij<}
   \tau^{i,j}\leq \bar cQ_0(q+1)\frac{\gw(s_i)}{\nu}, \q j_0\le i\le j
\end{equation}
where $\bar c=4c$ ($c$ as in \req{c_for_k}).

The proof is by induction. We apply Lemma \ref{tau1} in the case where
$$b \text{ is replaced by }b_{i+1,j},\q M'=M_i,\q M=M_j,\q j_0\leq i\le j.$$
For $i=j$ we  put $b_{j+1,j}:=b$.  Note that, for $M\geq M_{j_0}$,  condition \req{Mj0}(ii) yields \req{Mcon1}.

Applying Lemma \ref{tau1} and Lemma \ref{L2.1} to the case $i=j$ we obtain
$$ \exp\frac{\tau^{j,j}}{2cs_j}\leq C_5(b)M_j^{q+1-\nu\frac{H_0}{Q_0}}.$$
\Consy, using \req{M_j} and condition \req{Mj0}(iii)
\begin{equation}\label{tau/s}\BAL
 \frac{\tau^{j,j}}{2cs_j}&\leq \ln C_5(b) +\big(q+1-\nu\frac{H_0}{Q_0}\big)\ln M_j\\
 &\leq 2(q+1)\frac{Q_0\mu(s_j)}{\nu}.
\EAL\end{equation}
For the last inequality recall that $s_j=s_\nu(M_j)$, which implies,
$$\ln M_j=\frac{Q_0\mu(s_j)}{\nu}.$$
Inequality \req{tau/s} implies \req{tauij<} for $i=j$.

Observe that $s_j\dar 0$ as $j\uar\infty$ and \consy, $\gw(s_j)\dar 0$. Therefore if $j_0$ is sufficiently large we have  $\tau^{j,j}<b/2$ and $b_{j,j}>b/2$.
By Proposition \ref{IMtau},
\begin{equation}\label{main-1}
  I_{b_{j,j}}(M_j)\leq C_1(b_{j,j})M_{j}^{(q+1)/a}\leq C_1(b)M_{j-1}^{q+1}.
\end{equation}
Here we use condition \req{Mj0}(i) and the fact that $b_{j,j}=b-\tau_{b,\nu}(M_j,M_j)$.

Now we apply Lemma \ref{tau1} for $i=j-1$, i.e.,  when $b$ is replaced by $b_{j,j}$ and $M'=M_{j-1}$, $M=M_j$.
This lemma, combined with \req{main-1}, yields
$$ \BAL&\exp\frac{\tau^{j-1,j}}{2cs_{j-1}}\leq \\
 c_1&\Big(I_{b_{j,j}}(M_j)+C_3(b_{j,j})M_j^2M_{j-1}^{q-1}\Big)
C_2(b_{j,j})^{-1}M_{j-1}^{-\nu\frac{H_0}{Q_0}}\leq\\
c_1&\Big(C_1(b_{j,j})M_{j-1}^{q+1}+C_3(b_{j,j})M_j^2M_{j-1}^{q-1}\Big)C_2(b_{j,j})^{-1}M_{j-1}^{-\nu\frac{H_0}{Q_0}}.
\EAL$$
By \req{a=},
\begin{equation}\label{Mj2}
 M_j^2M_{j-1}^{-\nu\frac{H_0}{Q_0}}\leq M_{j-1}^2.
\end{equation}
Therefore, similarly to \req{tau/s}, we obtain
\begin{equation}\label{tau/s1}\BAL
 \frac{\tau^{j-1,j}}{2cs_{j-1}}&\leq \ln C_5(b_{j,j}) +(q+1)\ln M_{j-1}\\
 &\leq2(q+1)\frac{Q_0\mu(s_{j-1})}{\nu},
\EAL\end{equation}
which, in turn,  implies \req{tauij<} for $i=j-1$.

This process can be repeated inductively for $i=j-2, j-3,\ldots, j_0$ provided that $b_{i+1,j}\geq b/2$. For each value of $i$ in this range we first apply Proposition \ref{IMtau} to obtain,
\begin{equation}\label{main-i}
  I_{b_{i+1,j}}(M_j)\leq C_1(b_{i+1,j})M_{i+1}^{(q+1)/a}\leq C_1(b)M_{i}^{q+1}.
\end{equation}
After that we apply Lemma \ref{tau1} combined with \req{main-i} to obtain \req{tauij<} for the respective value of $i$, always with the same constant $\bar c$.
Therefore, to complete the proof, it remains to be shown that there exists $j_0$ \sth:

$\q$ \emph{If  $j>j_0$, $j_0\leq k<j$ and $\tau^{i,j}$ satisfies \req{tauij<} for $k\leq i\leq j$ then,}
\begin{equation}\label{sum<b/2}
    \sum_{i=k}^j \tau^{i,j} < b/2.
\end{equation}

\bcom@
We apply Lemma \ref{tau1} with $b$ replaced by $b_{i+1,j}$, $M'=M_i$, $M=M_j$.

Further, by induction, applying Proposition \ref{IMtau} to $u_j$ in $\Gw\ind{b_{i,j}}$, $i=j_0,\ldots,j-1$ we obtain
\begin{equation}\label{Mest}
  I_{i,j}:= \int_{\Gw_{b_{i,j}}}(|\nabla_xu_j|^2+h(x)u_j^{q+1})dx\leq   CM_{i-1}^{q+1}.
\end{equation}

 We observe that estimate \req{1.2.43} depends only on conditions \req{Mcon1} and \req{Mcon3}. Therefore, if $j_0$ satisfies \req{M1,3} then, by  Lemma \ref{taunu<},
 $$\tau_{i,j}\leq C(N,q,\nu)\gw(s_i), \q j_0\le i< j.$$
@ \end{comment}

By \req{snuj} and \req{tech1}
$$s_i\leq (Q_0/\nu)a^{-i}\gw(s_i)\leq \ell a^{-i},\q \ell:=Q_0\gw(s_0)/ \nu.$$
Since, by assumption,   \req{tauij<} holds for $k\leq i\leq j$,
$$\BAL
\sum_{i=k}^j \tau_{i,j}\leq C(N,q,\nu)\sum_{i=k}^j \gw(s_i)\leq  C(N,q,\nu)\sum_{i=k}^j \gw(\ell a^{-i})
\EAL$$
Further, using the monotonicity of $\gw$,
$$\sum_{i=k}^j \gw(\ell a^{-i})\leq \int_{k}^j \gw(\ell a^{-s})ds<\int_0^{\gb_k}\frac{\gw(r)}{r}dr$$
where $\gb_k=\ell a^{-k}$.
Because of the Dini condition, the last integral tends to zero when $\gb_k\to 0$. Therefore, if $j_0$ is sufficiently large (depending only on $N,q,\nu$ and $a$) \req{sum<b/2} holds for all $k\geq j_0$.
\end{proof}

\noindent\emph{Completion of proof of Theorem \ref{Th.1'}}.\hskip 2mm Since $U_M$ increases as $M$ increases
$$U^R:=\lim_{M\tin} U_M=\lim_{j\tin}u_j.$$

The function $V_M$ defined by
$$V_M(x)=U_M(x',x_N+R)$$ is a solution of \req{eqh} in the ball $B_R(x^R)$ where $x^R=(0,R)$.
If $v$ is a solution of \req{eqh} in $\RN_+$ then
$$v\leq V^R:=\lim_{M\tin}V_M \q \text{in }B_R(x^R).$$
It remains to prove that $V^R$ is bounded in a neighborhood of the point $(0,R)$ which is equivalent to $U^R$  being bounded in a neighborhood of the origin.

By interior elliptic estimates, \req{I_b/2} implies that
\begin{equation}\label{1.2.70}
\sup_{j_0\le j}\int_{\Omega_{b/3}}|u_j|^2dx<\infty.
\end{equation}
Since
$h(x)\geq0$, $u_j$ is subharmonic in $\Gw_b$. Therefore \req{1.2.70} implies
\begin{equation}\label{1.2.71}
 \sup\{u_j(x):\,j_0\le j,\; x\in \Gw_{b/4}\}<\infty.
\end{equation}
\qed

\section{Proof of Theorem \ref{Th.2}}

Put

$$r_j:=2^{-j},\qq \Gw_j=\{(x',x_N):\,|x'|<r_j,\, 0<x_N\}, \q j=1,2,\ldots\,.$$
\bcom@
For every $M>0$ denote by $w_{j,M}$ the solution of \req{eqh} in $\RN_+$ \sth
\begin{equation}\label{vjM}
  w_{j,M}(x',0)=\begin{cases} M &\text{if }|x'|<r_{j+1}\\ 0 &\text{otherwise.}\end{cases}
\end{equation}
We show that, if \req{tech1} and \req{nonDini} hold then there exists an increasing \seq $\{M_j\}$ tending to infinity \sth
\begin{equation}\label{vjMj}
 \lim_{j\tin} w\ind{j,M_j}(0,x_N)=\infty \forevery x_N>0.
\end{equation}
@\end{comment}
Further denote,
\begin{equation}\label{aj}
 a_j:=\exp\big( -\mu(r_j)\big), \q   A_j=\big( a_j r_j^2\big)^{\rec{q-1}}
\end{equation}
and, for $x'\in \BBR^{N-1}$,
\begin{equation}\label{ggj}
\gg_j(x')=\begin{cases}A_j^{-1}\gf_1(x'/r_{j+1}) &\text{if }|x'|<r_{j+1}\\
 0&\text{if }|x'|\geq r_{j+1}\end{cases}
\end{equation}
where $\gf_1$ the first eigenfunction of the Dirichlet problem to
$-\Gd_{y'}$ in $B_1^{N-1}$ normalized by $\gf_1(0)=1$. Recall that
$\mu(s)=\gw(s)/s$.

We consider the  \bvp{s}
\begin{equation}\label{bvpj}\BAL
 -\Gd u_j +a_ju_j^q&=0 &&\q\text{in }\Gw_j,\\
 u_j(x)&=0 &&\q\text{on }\{x\in \bdw_j: x_N>0\},\\
 u_j(x',0)&=\gg_j(x') &&\q\text{for } |x'|\leq r_j.
\EAL\end{equation}

 In view of \req{tech1}, $\{a_j\}$  is a  decreasing \seq converging to zero and
$$a_j=\sup_{s\in (0,r_j)} \exp\big( -\mu(s)\big).$$
Therefore, for every $x_N>0$, $\{u_j(0,x_N\}$ is an increasing \seq
and $u_j$ is a subsolution of the problem
\begin{equation}\label{bvpw}\BAL
 -\Gd w +h(x)w^q&=0 &&\q\text{in }\Gw_j,\\
 w(x)&=0 &&\q\text{on }\{x\in \bdw_j: x_N>0\},\\
 w(x',0)&=\gg_j(x') &&\q\text{for } |x'|\leq r_j.
\EAL\end{equation}

The proof of Theorem \ref{Th.2} is based on the following:
\begin{proposition}\label{uj_infty} For every $x_N>0$,
$$\lim_{j\tin}u_j(0,x_N)=\infty.$$
\end{proposition}

In the next lemma we collect several results of  Brada \cite{Brd} that are used in the proof of this proposition.

\begin{lemma}\label{Brada} Let $a$ be a positive number, let $q>1$ and let $f$ be a positive function in $L^\infty(B_1^{N-1})$, where $B_1^{N-1}$ denotes the unit ball in $\BBR^{N-1}$ centered at the origin.

Consider  the  problem
\begin{equation}\label{bvp0}\BAL
-\Gd u + b\, u^q&=0 &&\q \text{in } D_0\\
u(y)&=0 &&\q\text{for }y\in \prt D_0\,:\,0<y_N,\\
u(y',0)&=f(y') &&\q\text{for } |y'|\leq 1,
\EAL\end{equation}
where
$$D_0=\{y=(y',y_N)\in \RN: |y'|<1,\; 0<y_N\}.$$

If $u$ is the solution of this problem then there exists a number $\ga>0$ \sth
\begin{equation}\label{asymp1}
\lim_{y_N\to\infty}  \exp\big(\sqrt{\gl_1}y_N\big)u(y)=\ga\gf_1(y')
\end{equation}
uniformly in $B_1^{N-1}$. Here $\gl_1$ is the first eigenvalue and $\gf_1$ the corresponding eigenfunction of $-\Gd_{y'}$ in $B_1^{N-1}$ normalized by $\gf_1(0)=1$.

The limit $\ga$ satisfies
\begin{equation}\label{ga-est}
  \ga\leq c b^{-\rec{q-1}}\sup f.
\end{equation}
\end{lemma}
\proof By \cite[Theorem 4]{Brd}, \req{asymp1} holds for some $\ga\in \BBR$. Under our assumptions $u$ is positive so that $\ga\geq 0$. By the remark in \cite[p.357]{Brd}, if $\ga=0$ then there exists $k>1$ \sth
$$\lim_{y_N\tin} \exp\big(\sqrt{\gl_k}y_N\big)u(y)=\gf_k(y')$$
where  $\gf_k$ an eigenfunction of $-\Gd_{y'}$ in $B_1^{N-1}$ corresponding to the $k$-th eigenvalue.
However this is impossible because $\gf_k$ changes signs.  Thus $\ga>0$.

Inequality \req{ga-est} is a consequence of \cite[Proposition 1]{Brd}.
\qed
\subsection{An estimate of $u_j$.}
We start by rescaling problem \req{bvpj}. Put
\begin{equation}\label{scale}\BAL
  y=x/r_j,\q  \tl u_j(y) = A_ju_j(r_j y),
\EAL\end{equation} where $A_j$ is given by \req{aj}. Then $v:=\tl
u_j$ is the solution of the problem

\begin{equation}\label{bvpv}\BAL
 -\Gd v+v^q&=0 &&\q\text{in } D_0,\\
 v(y)&=0 &&\q\text{for }y\in \prt D_0\,:\,0<y_N,\\
v(y',0)&=\tl\gg(y') &&\q\text{for } |y'|\leq 1,
\EAL\end{equation}
where
\begin{equation}\label{tlgg}
 \tl\gg(y'):=\begin{cases}\gf_1(2y') &\text{if }|y'|<\rec{2}\\
 0&\text{otherwise.}\end{cases}
\end{equation}

Applying Lemma \ref{Brada} to the solution $v$ of \req{bvpv} we obtain,
\begin{equation}\label{limvj}
  \lim_{y_N\tin}\exp(\sqrt{\gl_1}y_N) v(y',y_N)=\ga \gf_1(y')
\end{equation}
where $\ga$ is a positive number depending only on $q,N$. \Consy there exists  $\gb>0$ \sth
$$\BAL \frac{1}{2}\ga \gf_1(y')\exp(-\sqrt{\gl_1}y_N)&\leq  A_j u_j(r_jy) \\
&\leq 2 \ga \gf_1(y')\exp(-\sqrt{\gl_1}y_N) \forevery y_N\geq \gb,\; |y'|\leq 1.
\EAL$$
This inequality is equivalent to
\begin{equation}\label{uj>}\BAL
&\frac{\ga}{2A_j}  \gf_1(x'/r_j)\exp(-\sqrt{\gl_1}x_N/r_j)\leq u_j(x)\\
&\leq \frac{2\ga}{ A_j}  \gf_1(x'/r_j)\exp(-\sqrt{\gl_1}x_N/r_j)
 \forevery x_N\geq \gb r_j,\q |x'|\leq r_j.
\EAL\end{equation}
\subsection{Comparison of $u_j$ and $u_{j-1}$.} \
Let $\tau_{j}$ be the number determined by the equation,
\begin{equation}\label{tauj}\BAL
  \frac{\ga}{2}  \exp(-\sqrt{\gl_1}\tau_{j}/r_j)=&\Big(\frac{a_j}{a_{j-1}}\Big)^{\rec{q-1}}2^{-\frac{2}{q-1}}\\
   =& 2^{-\frac{2}{q-1}} \exp\frac{-\mu(r_j)+\mu(r_{j-1})}{q-1}
\EAL\end{equation}
By \req{aj} and \req{ggj}, this  is equivalent to
\begin{equation}\label{tauj'}
\frac{\ga}{2 A_j}  \gf_1(x'/r_j)\exp\Big(-\sqrt{\gl_1}\frac{\tau_{j}}{r_j}\Big)=\gg_{j-1}(x').
\end{equation}

Without loss of generality we may assume that \req{tech1'} holds for $a=2$. Therefore
 there exists $\kappa\in (0,1)$ \sth
\begin{equation}\label{wrj}
 \mu(r_j)-\mu(r_{j-1})\geq \kappa \mu(r_j).
\end{equation}
By \req{tauj},
$$\sqrt{\gl_1} \frac{\tau_j}{r_j} =\frac{\mu(r_j)-\mu(r_{j-1})}{q-1} +c(N,q).$$
Therefore, by \req{wrj} and \req{tech1}, there exist positive numbers $c_0,\, c_1$ and $j_0$ (depending only on $\kappa,\,N,\,q$) \sth
\begin{equation}\label{tauj<}
\gb r_j <c_0\gw(r_j)\leq \tau_{j}\leq c_1\gw(r_j)
\end{equation}
for every $j\geq j_0$ ($\gb$  as in \req{uj>}).

By \req{uj>}, \req{tauj'} and \req{tauj<}
\begin{equation}\label{ggj1}
 \gg_{j-1}(x')\leq u_j(x',\tau_{j}), \q |x'|\leq r_j, \;j\geq j_0.
\end{equation}
\bcom@
Let $u_{j,1}$ be the solution of the problem
\begin{equation}\label{bvpj1}\BAL
 -\Gd u_{j,1} +a_{j-1}u_{j,1}^q&=0 &&\q\text{in }\Gw_j,\\
 u_{j,1}(x)&=0 &&\q\text{on }\{x\in \bdw_{j-1}: x_N>0\},\\
 u_{j,1}(x',0)&=\gg_{j,1}(x') &&\q\text{for } |x'|\leq r_{j-1},
\EAL\end{equation}
@\end{comment}
By  the maximum principle, \req{bvpj}, \req{ggj1} and the fact that $a_{j-1}>a_j$ imply
\begin{equation}\label{iterate0}
 u_{j-1}(x',x_N)\leq u_j(x',x_N+\tau_{j}) \forevery j\geq j_0,\; x\in \Gw_j.
\end{equation}

\subsection{Proof of Proposition \ref{uj_infty}.} \
Let $j_0\leq k<m$. Iterating inequality \req{iterate0} for $j=k+1,\ldots,m$ we obtain,
\begin{equation}\label{iterate1}
 u_{k}(x',x_N)\leq u_m(x',x_N+\sum_{j=k+1}^m\tau_{j}) \forevery   x\in \Gw_m.
\end{equation}
Combining this inequality (for $x'=x_N=0$) with  \req{uj>} yields
\begin{equation}\label{iterate3}
\rec{2}\ga(a_{k}r_{k}^2)^{-\rec{q-1}}=\frac{\ga}{ 2A_{k}}\leq u_{k}(0)\leq u_m(0,\sum_{j=k+1}^m\tau_{j}) \end{equation}
for every $m,k$ \sth $j_0\leq k<m$. By \req{nonDini},
$$\sum_{j=k}^\infty\gw(r_j)=\infty.$$
Therefore, by \req{tauj<}
\begin{equation}\label{sum_tau}
   \sum_{j=k}^\infty\tau_j=\infty.
\end{equation}
\Consy,
\begin{equation}\label{sm,k}
s_{m,k}:=\sum_{j=k+1}^m\tau_{j}\;\Lra\; \lim_{m\tin}s_{m,k}=\infty.
\end{equation}
Note that $a_kr_k^2\to 0$; therefore, by \req{iterate3}, for every $M>0$ there exists $j\ind{M}$ \sth
\begin{equation}\label{M<um}
  M<u_m(0,s_{m,k}) \q j\ind{M}\leq k<m.
\end{equation}

 We claim that
 \begin{equation}\label{U=infty}
 \sup u_j(0,x_N)=\infty\forevery x_N>0.
 \end{equation}
 By negation, assume that
$$\exists s>0\,:\;\sup u_j(0,s)= K<\infty.$$
By \req{uj>}
$$\frac{u_j(x',s)}{u_j(0,s)}\leq 4\ga \q |x'|\leq r_j.$$
Here we use the fact that $1=\gf(0)=\max\gf$. It follows that, for every $j$ \sth $2^j>\gb /s$,
$$\sup u_j(x',s)\le 4\ga K, \q |x'|\leq r_j. $$
Therefore, by the maximum principle, for every $j$ as above,
$$u_j(x',x_N)\le 4\ga K \forevery x\in \Gw_j\cap [x_N\geq s].$$
In view of \req{sm,k},
this contradicts \req{M<um}.
\qed

\subsection{Proof of Theorem \ref{Th.2}.} \

Let $P_0(x,y)=c_Nx_N|x-y|^{-N}$ be the Poisson kernel for $-\Gd$ in $\RN_+$.
  Condition \req{tech1''} implies that,  for any positive constants $a,R$
 \begin{equation}\label{tech3}
 \sup_{|x'|<R}|x'|^{-a}h(x)<\infty.
 \end{equation}
 For every $q>1$ choose $a>0$ \sth $q<(N+1+a)/(N-1)$. Then for every $R>0$,
  \begin{equation*}\label{subcr}
\int_{[|x|<R, \;0<x_N]} h(x)P_0^q(x,0))x_N dx<C_a\int_{[|x|<R, \;0<x_N]} |x|^a P_0^q(x,0))x_N dx <  \infty.
\end{equation*}
\Consy, for every $k>0$, the problem
$$\BAL -\Gd v+h(x) v^q&=0 &&\text{in }D_{0},\\
v&=0 && \text{on }\prt_\ell  D_0:= [|x'|=1,\;x_N>0],\\
v&=k\gd_0&&\text{on }[x_N=0]
\EAL$$
possesses a unique solution dominated by the supersolution $kP_0$ (see \cite{MVsub}).

The function
\begin{equation}\label{v0infty}
  v_{0,\infty}:=\lim_{k\tin} v_{0,k}\q \text{in }D_0
\end{equation}
is a solution of \req{eqh} in $D_0\cap[|x'|>0]$ but it  may blow up as $|x'|\to 0$.

Put
$$f(x_N)=\int_{|x'|<1}v_{0,\infty}(x',\bar x_N)dx'\forevery x_N>0.$$
If $f(a)<\infty$  for some $a>0$
 then $v_{0,\infty}$ is finite in $D_0\cap [x_N>a]$ so that $f(x_N)<\infty$ for every $x_N>a$. Thus
  \begin{equation}\label{fxN}
f(a)<\infty \;\text{ for some $a>0$ }\Lra  f(x_N)<\infty \forevery x_N\geq a.
 \end{equation}
Let
\begin{equation}\label{def_b}
  b=inf\{x_N>0: f(x_N)<\infty\}.
\end{equation}
By\req{fxN}
  \begin{equation}\label{fxN=infty}
f(x_N)=\infty\forevery x_N\in (0, b),\q f(x_N)<\infty \forevery x_N\in (b,\infty).
\end{equation}
We have to show that $b=\infty$. By negation assume that $b<\infty$. First consider the case $0<b$. Let $a\in (0,b)$ and put $\eta(x');=v_{0,\infty}(x',a)$. Then
$$\int_{|x'|<1}\vgf\eta\,dx'=\infty \forevery \vgf\in C([|x'|\leq 1] \q\text{\sth }\;\vgf(0)>0.$$
Thus the measure $\mu_\eta=\eta\,dx'$ is larger then $k\gd_0$ for every $k>0$. The function $V$ given by $V(x)=v_{0,\infty}(x',x_N+a)$ satisfies
$$\BAL -\Gd V+h(x) V^q&=0 &&\text{in }D_{0},\\
V&=0 && \text{on }\prt_\ell  D_0:= [|x'|=1,\;x_N>0],\\
V&=\eta &&\text{on }[x_N=0].
\EAL$$
Therefore $V\geq v_{0,\infty}$, i.e.,
$$v_{0,\infty}(x',x_N+a)\geq v_{0,\infty}(x',x_N).$$
But this implies
$$f(x_N+a)=\infty \forevery x_N\in (0,a+b)$$
which contradicts \req{def_b}.

Next assume that $b=0$. In this case,
\begin{equation}\label{v0<infty}
  v_{0,\infty}(0, x_N)<\infty \forevery x_N>0
\end{equation}
and \consy $v_{0,\infty}$ is a solution of \req{eqh} in $D_0$. Let $w_j$ be the unique solution of the \bvp:
\begin{equation}\label{bvp-aj}
\BAL
  -\Gd w_j+a_jw_j^q&=0, &&\text{in }\Gw_j\\
  w_j&=0 && \text{on } \prt \Gw_j\cap [x_N>0],\\
  w_j&=\infty\gd_0 && \text{on } [x_N=0].
  \EAL
\end{equation}
where $a_j=h(r_j)$. As usual, this means that $w_j=\lim_{k\tin} w_{j,k}$ where $w_{j,k}$ is the solution of the modified problem  where the boundary data on $x_N=0$ is $w_{j,k}(x',0)=k\gd_0$.
Since $a_j\geq h(x)$ in $\Gw_j$ it follows that
\begin{equation}\label{wj<v}
 w_j\leq v_{0,\infty} \q\text{in  }\Gw_j.
\end{equation}

The function $w^*_j$ given by $w^*_j(x):=A_jw_j(r_jx)$ for $x\in D_0$ is a solution of the problem:
\begin{equation}\label{bvp-q}
\BAL-\Gd w+w^q&=0 && \text{in }D_0\\
w&=0 &&\text{on }\prt_\ell D_0,\\
w(x',0)&=\infty\gd_0 &&\text{on }[x_N=0].
\EAL
\end{equation}
The solution of this problem is unique; \consy $w^*_j$ is independent of $j$ and we denote it by $w^*$.

Let $C:=\sup_{|x'|<1/2} w^*(x',1)$. Then
$w_j(y)=A_j^{-1}w^*(y/r_j)$ satisfies
$$ w_j(y',r_j)\geq cA_j^{-1}, \q |y'|<r_{j+1}.$$
As $\gg_j(x')=0$ for $|x'|>r_{j+1}$ it follows that
$$ w_j(y',r_j)\geq  c\gg_j(x'), \q |x'|<r_{j}.$$
Hence
$$w_j(x', x_N+r_j)\geq u_j(x) \q\text{in }\Gw_j.$$
Therefore, by  Proposition \ref{uj_infty},
$$\lim_{j\tin} w_j(0,x_N)=\infty \forevery x_N>0.$$
Hence, by \req{wj<v},
$$v_{0,\infty}(0,x_N)=\infty \forevery x_N>0$$ in contradiction to \req{v0<infty}.

\qed


\begin{thebibliography}{99}




\bibitem{BanM} Bandle C. and Marcus M., \emph{Asymptotic behavior of solutions and their derivatives, for semilinear,
        elliptic problems with blowup at the boundary,} Ann. Inst. H. Poincare, v.~12 (1995), p.155-171.
\bibitem{Brd} Brada A., \emph{Comportement asymptotique de solutions d'\'equations elliptiques semi-lin\'eares dans un cylindre,} Asymptotic An. v. 10 (1995), p. 335-366.
\bibitem{Keller}Keller J. B., \emph{On solutions of $\Delta u = f(u)$,} Comm. Pure Appl. Math., v. 10 (1957), 503-510.
\bibitem{MVsub} Marcus M. and V\'eron L., \emph{The boundary trace of positive solutions of semilinear elliptic equations:
         The subcritical case,} Arch. Rat. Mech. Anal., v. 144 (1998), p. 201-231.
\bibitem{MV-ANS} Marcus M. and V\'eron L., \emph{Initial trace of
positive solutions to semilinear parabolic inequalities,} Adv.
Nonlinear Studies, v.2 (2002) p.395-436.
\bibitem{MV-Pisa} Marcus M. and V\'eron L., \emph{Boundary Trace of Positive Solutions of Nonlinear
Elliptic Inequalities,} Ann. Sc. Norm. Sup. Pisa, V. III (2004) p.481-533.
\bibitem{Oss} Osserman R., \emph{On the inequality $\Delta u \geq f(u)$,} Pacific J. Math., v. 7 (1957) pp. 1641-1647.
\bibitem{ShV1}Shishkov A. and V\'eron L., \emph{The balance between diffusion and absorption
in semilinear parabolic equations,} Rend. Lincei Mat. Appl., v. 18 (2007), 59–96.
\bibitem{ShV2}Shishkov A. and V\'eron L., \emph{Diffusion versus absorption in semilinear elliptic equations,} J. Math. Anal. Appl., v. 352 (2009) 206–217.
\bibitem{ShV3}Shishkov A. and V\'eron L., \emph{Singular solutions of some nonlinear parabolic equations with spatially inhomogeneous absorption,} Calc. Var. P. D. E., v. 33 (2008) 343-375.
\end{thebibliography}
\end{document}